\documentclass[12pt,reqno]{amsart}
\usepackage[utf8]{inputenc}
\usepackage{amsfonts,amssymb,amsmath,amsthm}
\usepackage{times}
\usepackage{fullpage}
\usepackage{xcolor}
\usepackage[colorlinks=true,citecolor=blue,linkcolor=blue]{hyperref}

\newtheorem{theorem}{Theorem}[section]
\newtheorem{proposition}[theorem]{Proposition}
\newtheorem{lemma}[theorem]{Lemma}
\newtheorem{corollary}[theorem]{Corollary}

\newtheorem{claim}[theorem]{Claim}

\newtheorem*{claim*}{Claim}

\theoremstyle{definition}

\newtheorem{problem}[theorem]{Problem}

\newcommand{\F}{{\mathbb F}}
\newenvironment{poc}{\begin{proof}[Proof of claim]}{\end{proof}}

\numberwithin{equation}{section}
\numberwithin{figure}{section}
\numberwithin{table}{section}

\title{Paley-type matrices and $1$-factorizations of complete graphs}
\author{Chi Hoi Yip}
\address{Department of Mathematics, Hong Kong University of Science and Technology, Clear Water Bay, Hong Kong}
\email{machyip@ust.hk}
\author{Semin Yoo}
\address{Discrete Mathematics Group \\ Institute for Basic Science \\ 55 Expo-ro Yuseong-gu, Daejeon 34126 \\ South Korea}
\email{syoo19@ibs.re.kr}
\author{Shikang Yu}
\address{School of Mathematics and Statistics, Beijing Jiaotong University, Beijing, 100044, PR China}
\email{healthyu@bjtu.edu.cn}
\subjclass[2020]{Primary 05C70; Secondary 05B20, 11T24, 20K01}
\keywords{$1$-factorization, Hadamard matrix, Paley matrix, quadratic residue, orthomorphism, finite abelian group}

\begin{document}

\begin{abstract}
Ball, Ortega--Moreno, and Prodromou asked two questions about whether, for every odd prime $p$, one can find a $1$-factor of the complete graph $K_{p+1}$ with some arithmetic restrictions related to quadratic residues. These problems are motivated by two natural compatibility conditions between $1$-factorizations and the sign patterns of certain Paley-type matrices. Recently, Afifurrahman et al. made some partial progress on the second problem. In this paper, we completely resolve both problems. We prove that the first problem has a solution precisely when $p\equiv3\pmod4$, while the second problem has a solution for every
odd prime $p$. We also solve a further problem of Ball et al. for cyclic groups of odd order, and more generally for all finite abelian groups of odd order.
\end{abstract}

\maketitle

\section{Introduction}
Throughout the paper, let $p$ be a prime. Let $\mathbb{F}_p$ be the finite field with $p$ elements, and $\F_p^*=\F_p \setminus \{0\}$.

Let $n$ be a positive integer and let $K_n$ denote the complete graph on $n$ vertices.
A \emph{$1$-factor} of $K_{n}$ (of even order) is a set of independent edges spanning the vertices of $K_n$, and a \emph{$1$-factorization} of $K_n$ is a partition of the set of edges of $K_n$ into $1$-factors. 
It is a classical result that $K_n$ admits a $1$-factorization whenever $n$ is even. It is therefore natural to ask for $1$-factorizations satisfying additional combinatorial constraints; see, for example, the survey \cite{MR85} and the monograph \cite{W97}. 

A natural source of additional constraints on a $1$-factorization is obtained by prescribing, for each factor, a balanced bipartition of the vertex set and requiring every edge of the factor either to lie across it or to have both endpoints in one of its parts. Hadamard matrices provide a particularly symmetric source of such bipartitions; see \cite[Chapter~2]{H07} for background on Hadamard matrices and \cite{CD06} for related design-theoretic material. Indeed, after normalization, the first row is constant, and orthogonality to this row forces each non-first row to have equally many $1$ and $-1$ entries. Thus each non-first row determines a balanced bipartition of the columns according to the signs of its entries, while mutual row orthogonality imposes strong pairwise uniformity among these bipartitions. This correspondence is also numerically well suited to $1$-factorizations: a $1$-factorization of $K_n$ consists of exactly $n-1$ factors, while a normalized Hadamard matrix of order $n$ has exactly $n-1$ non-first rows. Thus Hadamard matrices naturally provide balanced sign constraints for $1$-factorizations.

In 2019, Ball, Ortega--Moreno, and Prodromou \cite{BOP19} initiated the study of $1$-factorizations that are compatible with the sign patterns of normalized Hadamard matrices.
In particular, they \cite[Section~2]{BOP19} proposed an arithmetic viewpoint that connects such factorizations with the celebrated Paley construction of Hadamard matrices \cite{P33}. 

A \emph{Hadamard matrix} of order $n$ is a square matrix $H$ whose entries are either $1$ or $-1$ and whose rows are orthogonal, i.e., $HH^T=nI$.
Let $p$ be an odd prime. Consider the matrix
\[
L_p=\Big(\Big(\frac{j-i}{p}\Big)\Big)_{i,j\in\mathbb{F}_p},
\]
where $\big(\frac{\cdot}{p}\big)$ denotes the Legendre symbol.
Furthermore, we define the matrix $H_p$ by
\[
H_p=\begin{bmatrix}
0 & \mathbf{1} \\
-\mathbf{1}^T & L_p
\end{bmatrix}+I,
\]
where $\mathbf{1}=(1,1,\dots,1)\in\mathbb{R}^p$ is a row vector and $I$ is the identity matrix of order $p+1$.
When $p\equiv 3\pmod 4$, the matrix $H_p$ is known as a \emph{Paley matrix}, and Paley \cite{P33} proved that $H_p$ is a Hadamard matrix. Paley-type matrices form one of the standard infinite families in the theory of Hadamard matrices. In the present setting they are particularly natural: their sign patterns are encoded by the quadratic character on $\F_p$, so the two compatibility conditions for $1$-factorizations become explicit arithmetic problems involving quadratic residues and quadratic non-residues. Write
\[
Q:=\{x^2:x\in\F_p^*\},\qquad NQ:=\F_p^*\setminus Q
\]
for the sets of nonzero quadratic residues and quadratic non-residues, respectively.

Following \cite{BOP19}, let $K_{p+1}$ be the complete graph with vertex set $\F_p \cup \{c\}$, where $c$ is an additional point, called the \emph{centre}. We use $\F_p \cup \{c\}$ to label the indices of the columns of $H_p$, where $c$ corresponds to the first column, and assume that the indices of the rows of $H_p$ start from $0$ and end with $p$.

Ball, Ortega--Moreno, and Prodromou \cite{BOP19} studied two ways for $1$-factorizations $\{F_1,\dots,F_p\}$ of $K_{p+1}$ to be compatible with the sign pattern of $H_p$: if an edge $e=\{i,j\}$ belongs to the factor $F_k$, then the entries in row $k$ of $H_p$ at columns $i$ and $j$ are required either to have opposite signs, denoted by \textup{(R1)}, meaning $(H_p)_{ki}(H_p)_{kj}<0$, or to have the same sign, denoted by \textup{(R2)}, meaning $(H_p)_{ki}(H_p)_{kj}>0$. To elaborate on their idea, we introduce some terminology following \cite{BOP19}. Given two distinct vertices $x$ and $y$ in $\F_p$, the \textit{length} of an undirected edge that connects $x$ and $y$, $\ell(\{x,y\})$, is the set of the two possible differences of the labels of the incident vertices. In other words, 
$$
\ell(\{x,y\})=\{x-y, y-x\}\subset \F_p^*.
$$
For convenience, we denote $\ell(\{x,y\})$ by $\ell(x,y)$. 
For edges incident to the centre $c$, we set
$\ell(c,x) := \infty$ for any $x\in \F_p$ so that an edge incident to $c$ can never share its length with an edge inside $\F_p$.

A key observation in \cite{BOP19} is that the rows of the matrix $L_p$ are cyclic permutations of the first row.
This suggests a cyclic method. Starting from a single $1$-factor $F_1$ on $\mathbb{F}_p\cup\{c\}$ corresponding to row $1$ of $H_p$, one hopes to obtain an
entire $1$-factorization by translating that $1$-factor through $\mathbb{F}_p$ while fixing $c$.
For such a translation to cover each edge exactly once, it is equivalent to the following condition: $F_1$ must contain, within $\F_p$,
exactly one edge $\{i,j\}$ of each possible length
\[
i-j=\pm1, \pm2, \dots, \pm\frac{p-1}{2},
\]
with edges incident to $c$ assigned $\infty$.
Thus, this viewpoint forces a purely combinatorial requirement that all selected edges in $F_1$ have pairwise distinct lengths. On the other hand, note that $(H_p)_{1,c}=-1, (H_p)_{1,0}=1$, and $(H_p)_{1,j}=\big(\frac{j}{p}\big)$ for each $j\in \F_p^*$. Thus, under restriction \textup{(R1)}, each edge in $F_1$ has one endpoint in $Q\cup\{0\}$ and the other in $NQ\cup\{c\}$, while under restriction \textup{(R2)}, both endpoints lie in the same one of these two sets. These observations led to the formulation of the following two problems.

\begin{problem}\label{prob:P1}[\cite{BOP19}, Problem 2.2 (P1)]
Let $p$ be an odd prime. Is there a $1$-factor $F$ of $K_{p+1}$ with vertex set $\F_p \cup \{c\}$ such that
\begin{enumerate}
\item both endpoints of each edge in $F$ lie in $Q\cup\{0\}$ or both lie in $NQ\cup\{c\}$, and
\item the lengths of the edges in $F$ are pairwise distinct?
\end{enumerate}
\end{problem}

\begin{problem}\label{prob:main}[\cite{BOP19}, Problem 2.2 (P2)]
Let $p$ be an odd prime. Is there a $1$-factor $F$ of $K_{p+1}$ with vertex set $\F_p \cup \{c\}$ such that 
\begin{enumerate}
\item each edge in $F$ has one endpoint in $Q\cup\{0\}$ and the other in $NQ\cup\{c\}$, and
\item the lengths of the edges in $F$ are pairwise distinct?
\end{enumerate}
\end{problem}

Several cases of these problems were already known. For Problem~\ref{prob:P1}, Ball et al. \cite{BOP19} obtained an affirmative answer when $p\equiv 7\pmod 8$ under the additional assumption that $2$ is a ``near-primitive root" modulo $p$, that is, $2\equiv x^2\pmod p$ for some primitive root $x$ modulo $p$. There is also an immediate parity obstruction: in any solution, $Q\cup\{0\}$ must be partitioned into pairs, so $(p+1)/2$ is even. Hence necessarily $p\equiv3\pmod4$.\footnote{In \cite[Theorem~2.3]{BOP19}, an affirmative answer was also stated for $p\equiv1\pmod4$. The proposed construction pairs each nonzero element $r$ with $-r$; these pairs have the required types and distinct lengths, but they leave $0$ and $c$ unmatched, while the remaining edge $\{0,c\}$ has one endpoint in each part.}

For Problem~\ref{prob:main}, Ball et al. \cite{BOP19} solved the case $p\equiv 3\pmod 4$. Their construction is quite simple: pair $d$ with $-d$ for each $d\in\{1,\ldots,(p-1)/2\}$ and pair $0$ with the centre $c$. Since $-1\in NQ$ in this case, each edge $\{d,-d\}$ has one endpoint in $Q$ and the other in $NQ$, and the edge lengths are pairwise distinct.
They also obtained a solution for primes $p\equiv 5\pmod 8$ under the additional assumption that $2$ is a primitive root modulo $p$.\footnote{It is not known unconditionally whether either of these additional hypotheses holds for infinitely many primes $p$: that $2$ is a primitive root modulo $p$, or that $2$ is a ``near-primitive root" modulo $p$. Both are instances of Artin-type conjectures.}
More recently, Afifurrahman, Primaskun, Etriana Putri, and Wijaya \cite{A25} resolved Problem~\ref{prob:main} for all primes $p\equiv5\pmod8$ via an explicit construction, removing the additional assumption on $2$ being a primitive root. Thus the remaining case was $p\equiv1\pmod8$, for which the methods of \cite{A25,BOP19} do not appear to extend.

We completely resolve both problems.

\begin{theorem}\label{thm:P1}
Problem~\ref{prob:P1} has an affirmative answer if and only if $p\equiv 3\pmod 4$.
\end{theorem}

\begin{theorem}\label{thm:main}
Problem~\ref{prob:main} has an affirmative answer for every odd prime $p$.
\end{theorem}

Together, Theorems~\ref{thm:P1} and~\ref{thm:main} completely resolve  \cite[Problem~2.2]{BOP19}. The necessity in Theorem~\ref{thm:P1} is exactly the parity observation above; for sufficiency we give a simple explicit construction pairing consecutive squares and negatives of consecutive squares. For Theorem~\ref{thm:main}, it remains only to handle $p\equiv1\pmod8$. We first give an explicit construction under a quartic condition on a primitive root, then use character-sum estimates, in particular, Weil's bound, to establish this condition for all sufficiently large $p$, and finally verify the remaining primes computationally.

Consequently, denote by $F$ such a $1$-factor of $K_{p+1}$ satisfying the conditions in Problem~\ref{prob:main}. By the above discussion,
translating $F$ through $\F_p$ (while fixing $c$) yields a $1$-factorization of $K_{p+1}$ satisfying restriction \textup{(R1)}. Similarly, when $p\equiv 3\pmod 4$, translating a solution to Problem~\ref{prob:P1} yields a $1$-factorization satisfying restriction \textup{(R2)}.

\begin{corollary}
Let $p$ be an odd prime. Then there exists a $1$-factorization of $K_{p+1}$ satisfying restriction \textup{(R1)} with respect to $H_p$. If $p\equiv 3\pmod 4$, then there also exists one satisfying restriction \textup{(R2)}.
\end{corollary}

Ball et al. also asked a natural generalization of Problem~\ref{prob:P1}. Their \cite[Problem~2.7]{BOP19} is stated for cyclic groups of odd order. In view of the surrounding $1$-factorization setting, the selected edges are naturally understood to be pairwise disjoint, so that they cover all but one vertex. They remarked immediately afterwards that they knew no counterexample even if the cyclic group were replaced by an arbitrary finite group of odd order. We formulate the corresponding question for finite abelian groups below. For an edge $\{x,y\}$ with $x,y\in G$, its two lengths are $\pm(x-y)$, that is, $x-y$ and $y-x$.

\begin{problem}\label{prob:general}
Let $G$ be a finite abelian group of odd order, written additively, and let $G=A\sqcup B$ be a partition. Is it always possible to find a set of $(|G|-1)/2$ pairwise disjoint edges $\{x,y\}$ with $x,y\in G$, whose $|G|-1$ lengths $\pm(x-y)$ include each nonzero element of $G$ exactly once and so that each edge joins either two elements of $A$ or two elements of $B$?
\end{problem}

We answer this question affirmatively.

\begin{theorem}\label{thm:general}
Problem~\ref{prob:general} has an affirmative answer.
\end{theorem}

The proof of Theorem~\ref{thm:general} uses orthomorphisms, classical objects in combinatorial design theory that are closely related to complete mappings and orthogonal Latin squares based on groups; see the classical paper \cite{JDM61} and, for a modern account, \cite[Section~1.3]{Evans18}. Here a permutation $\pi$ of a finite abelian group $G$ is an \emph{orthomorphism} if the map $x\mapsto\pi(x)-x$ is also a permutation of $G$. The key step is to prove that the number of orthomorphisms preserving both $A$ and $B$ is odd.

Afifurrahman et al.~\cite[Open Problem~2]{A25} also posed a cross-part analogue of Problem~\ref{prob:general} for finite abelian groups of odd order, in the spirit of Problem~\ref{prob:main}; our methods do not seem to extend directly to this variant.

Theorem~\ref{thm:general}, together with the elementary parity obstruction above, also gives another proof of Theorem~\ref{thm:P1}. Indeed, suppose that $p\equiv3\pmod4$ and apply Theorem~\ref{thm:general} to $\F_p=(Q\cup\{0\})\sqcup NQ$: the resulting pairwise disjoint edges leave exactly one vertex of $\F_p$ uncovered. Since $|Q\cup\{0\}|=(p+1)/2$ is even and $|NQ|=(p-1)/2$ is odd, this vertex lies in $NQ$. Pairing it with the centre $c$ gives a solution to Problem~\ref{prob:P1}. We nevertheless retain the direct proof of Theorem~\ref{thm:P1} in Section~\ref{sec:P1}, since it provides a simple explicit construction.

\smallskip
The rest of the paper is organized as follows. We prove Theorem~\ref{thm:P1} in Section~\ref{sec:P1}, Theorem~\ref{thm:main} in Sections~\ref{sec2}--\ref{sec3}, and Theorem~\ref{thm:general} in Section~\ref{sec:problem27}.

\section{The same-sign problem \textup{(P1)}}\label{sec:P1}

\begin{proof}[Proof of Theorem~\ref{thm:P1}]
The necessity was established by the parity observation in the introduction, so it remains to prove sufficiency. Suppose that $p\equiv 3\pmod 4$. Since $-1\in NQ$,
\[
Q\cup\{0\}=\left\{0^2,1^2,\ldots,\left(\frac{p-1}{2}\right)^2\right\}
\]
and
\[
NQ\cup\{c\}=\left\{-1^2,-2^2,\ldots,-\left(\frac{p-1}{2}\right)^2,c\right\}.
\]
Define
\begin{align*}
F:= {}&\left\{\{k^2,(k+1)^2\}:0\le k\le \frac{p-3}{2},\ k\equiv 0\pmod 2\right\}\\
&\cup\left\{\{-k^2,-(k+1)^2\}:0\le k\le \frac{p-3}{2},\ k\equiv 1\pmod 2\right\}\\
&\cup\left\{\left\{-\left(\frac{p-1}{2}\right)^2,c\right\}\right\}.
\end{align*}
It is immediate from the displayed descriptions of $Q\cup\{0\}$ and $NQ\cup\{c\}$ that $F$ is a $1$-factor satisfying the first condition in Problem~\ref{prob:P1}.

It remains to verify that its edge lengths are pairwise distinct. They are $\infty$ and
\[
\{\pm(2k+1)\},\qquad 0\le k\le \frac{p-3}{2}.
\]
Suppose that $0\le i<j\le (p-3)/2$ and
\[
\{\pm(2i+1)\}=\{\pm(2j+1)\}.
\]
Then either $2i+1\equiv 2j+1\pmod p$ or $2i+1\equiv -(2j+1)\pmod p$. In the first case, $p\mid j-i$; in the second, $p\mid i+j+1$. Both are impossible because
\[
0<j-i<p\qquad\text{and}\qquad 0<i+j+1<p.
\]
Thus the edge lengths are pairwise distinct, completing the proof.
\end{proof}

\section{A construction for the opposite-sign problem \textup{(P2)}}\label{sec2}
In this section, we construct a $1$-factor of $K_{p+1}$ when $p \equiv 1 \pmod 8$ to solve Problem~\ref{prob:main}. The section is devoted to the proof of the following proposition.

\begin{proposition}\label{prop:main}
Let $p\equiv 1\pmod 8$ be a prime. If there is a primitive root $a$ in $\F_p^*$ such that $a(a^2-a+1)\in(\F_p^*)^4=\{x^4: x \in \F_p^*\}$, then the answer to Problem~\ref{prob:main} is affirmative.
\end{proposition}

Let $p$ be an odd prime with $p\equiv 1\pmod 8$ and set
\[
M:=\frac{p-1}{2}.
\]
Then $M\equiv 0\pmod 4$ and $p-1=2M$.
Fix a primitive root $a\in\F_p^*$, so
\[
a^{M}=a^{(p-1)/2}\equiv -1\pmod p.
\]
We label vertices of $K_{p+1}$ as 
\[V(K_{p+1})=\{0,c,1,a,\ldots,a^{p-2}\}.\]
We view the exponents of $a$ modulo $p-1=2M$, so that $a^{t+2M}=a^t$ for all integers $t$. In particular, $a^{2M}=1=a^0$.

We now construct a $1$-factor of $K_{p+1}$.
First, let us pair $0$ with $c$.
To pair other vertices, we define the following four index sets
\begin{align*}
I_1&:=\{i: \ i\equiv 1 \pmod 4, \quad 1\le i\le M-3 \},\\
I_2&:=\{i:  i\equiv 3\pmod 4, \quad 3\le i\le M-1\},\\
I_3&:=\{i: \ i\equiv 1\pmod 4, \quad M+1\le i\le 2M-3\},\\
I_4&:=\{i: i\equiv 3\pmod 4, \quad M+3\le i\le 2M-1\}.
\end{align*}
Consider the edge set $\mathcal{M}$ consisting of
\begin{align*}
E_1(i)&:=\{a^i,a^{i+M+1}\} \quad \text{for each } i\in I_1,\\
E_2(i)&:=\{a^i,a^{i+M-3}\} \quad \text{for each }i\in I_2,\\
E_3(i)&:=\{a^i,a^{i+M-1}\} \quad \text{for each }i\in I_3,\\
E_4(i)&:=\{a^i,a^{i+M-1}\} \quad \text{for each }i\in I_4. 
\end{align*}
Next, we show that the above construction satisfies the required conditions in Problem~\ref{prob:main} by establishing the following three claims. 

\smallskip
First, we show that $\mathcal{M}$ is indeed a $1$-factor on the vertex set $$V(K_{p+1})\setminus\{0,c\}=\{1, a^1,a^2,\dots,a^{2M-1}\}.$$ Equivalently, we show that every vertex in $V(K_{p+1})\setminus\{0,c\}$ is used exactly once in $\mathcal{M}$.

\begin{claim}\label{lem:all vertices are used}
Let $M=(p-1)/2$.
\begin{enumerate}
\item If $i \in I_1$, then $i+M+1$ lies in the second half $\{M,\dots,2M-1\}$, and satisfies
\[i+M+1\equiv 2\pmod 4.\]
\item If $i \in I_2$, then $i+M-3$ lies in the second half $\{M,\dots,2M-1\}$, and satisfies
\[i+M-3\equiv 0\pmod 4.\]
\item If $i \in I_3$, then
$(i+M-1)\bmod 2M = i-M-1$,
and this lies in the first half $\{0,\dots,M-1\}$, and satisfies
\[ i-M-1\equiv 0\pmod 4.\]
\item If $i \in I_4$, then $(i+M-1)\bmod 2M = i-M-1$,
and this lies in the first half $\{0,\dots,M-1\}$, and satisfies
\[i-M-1\equiv 2\pmod 4.\]
\end{enumerate}
\end{claim}

\begin{proof}
(1) Since $i\ge 1$, we have $i+M+1\ge M+2\ge M$, and since $i\le M-3$, we have
$i+M+1\le 2M-2\le 2M-1$, thus $i+M+1$ is in the second half.
Since $i \equiv 1 \pmod 4$ and $M \equiv 0 \pmod 4$, we have 
$i+M+1\equiv 2\pmod 4$.

(2) Similarly, $i+M-3\ge M$ and $i+M-3\le 2M-4$, so it lies in the second half.
Since $i \equiv 3 \pmod 4$ and $M \equiv 0 \pmod 4$, we have 
$i+M-3\equiv 0\pmod 4$.

(3) Since $i\le 2M-3$, $i-M-1\le M-4$. Also, since $i\ge M+1$, we have $i-M-1\ge0$, so it lies in the first half. Using $i\equiv 1\pmod 4$ and $M\equiv 0\pmod 4$, we have $i-M-1\equiv 0\pmod 4$.

(4) The range $M+3\le i\le 2M-1$
implies $2\le i-M-1\le M-2$, so it lies in the first half. The conditions $i\equiv 3\pmod 4$ and $M\equiv 0\pmod 4$ imply
$i-M-1\equiv 2\pmod 4$.
\end{proof}

Since the images of $I_1, I_2, I_3$ and $I_4$ under these shifts lie in disjoint congruence classes modulo $4$ within the first and second half ranges, no vertex is repeated.
In addition, $|I_1|=|I_2|=|I_3|=|I_4|=M/4$, so $|\mathcal M|=|I_1|+|I_2|+|I_3|+|I_4|=M$.
Therefore, $\mathcal{M}$ is a $1$-factor on $V(K_{p+1})\setminus\{0,c\}$.

We next verify the first condition in Problem~\ref{prob:main}.

\begin{claim}\label{lem:QNR}
Each edge in $\mathcal{M}$ has one endpoint in $Q$ and the other in $NQ$.
\end{claim}

\begin{proof}
Since $a$ is a primitive root, $a^j\in Q$ if and only if $j$ is even.
By construction, the indices in $I_1, I_2, I_3, I_4$ are all odd, and thus the first endpoint of each edge $E_k(\cdot)$ with $k=1,2,3,4$ lies in $NQ$.
On the other hand, Claim~\ref{lem:all vertices are used} shows that the second endpoint indices are all even. Thus, they all lie in $Q$.
\end{proof}

Together with the edge $\{0,c\}$, this shows condition (1) in Problem~\ref{prob:main} holds for the entire 1-factor. It remains to verify condition (2) in Problem~\ref{prob:main} that the edge lengths in $\mathcal{M}$ are pairwise distinct under the extra assumption made in Proposition~\ref{prop:main}.

\begin{claim}
If $a(a^2-a+1)\in(\F_p^*)^4$, then all edge lengths in $\mathcal{M}$ are pairwise distinct.
\end{claim}

\begin{proof}
We verify that the lengths of all edges in $\mathcal{M}$ are pairwise distinct.
Note that, for two edges $\{a^m,a^{m+U}\}$ and $\{a^n,a^{n+V}\}$ with integers $m,n,U,V$, equality of their lengths is equivalent to
\begin{equation}\label{eq:length}
a^{m-n}=\pm\frac{1-a^V}{1-a^U}.    
\end{equation}
Since $a^M=-1$, we have 
\begin{align}
\frac{1-a^{M-1}}{1-a^{M+1}}&=\frac{1+\frac{1}{a}}{1+a}=\frac{1}{a},\label{eq:1}\\
\frac{1-a^{M-1}}{1-a^{M-3}}&=\frac{1+\frac{1}{a}}{1+\frac{1}{a^3}}
=\frac{a^2}{a^2-a+1},\label{eq:2}\\
\frac{1-a^{M-3}}{1-a^{M+1}}&=\frac{1+\frac{1}{a^3}}{1+a}
=\frac{a^2-a+1}{a^3}. \label{eq:3}
\end{align}

Let $1\leq k \leq 4$. First, we show that all edge lengths in $E_k$ are pairwise distinct. If $\ell(E_k(i))=\ell(E_k(j))$ for $i,j\in I_k$, then equation~\eqref{eq:length} implies that $a^{i-j}=\pm1$, so $i-j\equiv 0$ or $M\pmod{2M}$. However, since $i,j\in I_k$, we have $|i-j|<M$. Thus, we must have $i=j$.

Next, we compare edges belonging to two different families $E_k$ in the following four cases.

\emph{Between $E_1$ and $E_2$}:
If $\ell(E_1(i))=\ell(E_2(j))$ for $i\in I_1$ and $j\in I_2$, then equations~\eqref{eq:length} and~\eqref{eq:3} give
\[
a^{i-j}=\pm\frac{a^2-a+1}{a^3}.
\]
Since $i\equiv 1 \pmod 4$ and $j\equiv 3 \pmod 4$, we have $i-j\equiv 2\pmod4$. Thus, the left-hand side lies in $(\F_p^*)^2\setminus(\F_p^*)^4$.
On the other hand, by the assumption, $(a^2-a+1)/a^3$ is a fourth power.
Since $p\equiv 1 \pmod8$, it follows that $-1\in(\F_p^*)^4$ and thus $-(a^2-a+1)/a^3$ is also a fourth power. Thus, both elements on the right-hand side belong to $(\F_p^*)^4$, a contradiction.
\smallskip

\emph{Between $E_1$ and $E_3$, or between $E_1$ and $E_4$}:
Suppose that $\ell(E_1(i))=\ell(E_3(j))$ for $i\in I_1$ and $j\in I_3$, or $\ell(E_1(i))=\ell(E_4(j))$ for $i\in I_1$ and $j\in I_4$. In both cases, by equations~\eqref{eq:length} and~\eqref{eq:1}, we have
\[
a^{i-j}=\pm\frac1a.
\]
The right-hand side lies in $NQ$, since $a$ is primitive and $p \equiv 1 \pmod 4$.
But $i-j$ is even, so $a^{i-j}\in Q$, a contradiction. 
\smallskip

\emph{Between $E_2$ and $E_3$, or between $E_2$ and $E_4$}:
Suppose that $\ell(E_2(i))=\ell(E_3(j))$ for $i\in I_2$ and $j\in I_3$, or $\ell(E_2(i))=\ell(E_4(j))$ for $i\in I_2$ and $j\in I_4$. In both cases,  using equations~\eqref{eq:length} and~\eqref{eq:2}, we get
\[
a^{i-j}=\pm\frac{a^2}{a^2-a+1}.
\]
By the assumption, $a(a^2-a+1)\in Q$, so $\frac{a^2}{a^2-a+1}\in NQ$. Since $p\equiv 1 \pmod 8$, we have $-1\in Q$, and hence the right-hand side of the above equation lies in $NQ$. On the other hand, $i-j$ is even, so the left-hand side lies in $Q$, a contradiction. 
\smallskip

\emph{Between $E_3$ and $E_4$}:
If $\ell(E_3(i))=\ell(E_4(j))$ for $i\in I_3$ and $j\in I_4$, then equation~\eqref{eq:length} implies that $a^{i-j}=\pm1$,
so $i-j\equiv 0$ or $M\pmod{2M}$.
But $i\equiv 1\pmod4$ and $j\equiv 3\pmod4$, so $i-j\equiv 2\pmod4$, which is impossible since
$0$ and $M$ are both multiples of $4$. 
\end{proof}

Combining the above claims finishes the proof of Proposition~\ref{prop:main}.

\section{Primitive roots satisfying the quartic condition}\label{sec3}

In this section, we use character sum estimates to guarantee the existence of a primitive root satisfying the quartic condition in Proposition~\ref{prop:main}, completing the proof of Theorem~\ref{thm:main}.
Throughout the section, we assume that $p\equiv 1 \pmod 8$ is a prime.

\subsection{Preliminaries}
Throughout the section, we use the following standard number-theoretic functions:
\begin{itemize}
    \item For each positive integer $n$, let $\tau(n)$ be the number of divisors of $n$ and $\omega(n)$ be the number of distinct prime divisors of $n$. 
    \item For each positive integer $n$, let $\varphi(n)$ be the number of positive integers at most $n$ that are coprime with $n$.
        \item For each prime $p$, let $\mathcal{G}(p)$ denote the set of primitive roots in $\F_p$. Note that $|\mathcal{G}(p)|=\varphi(p-1)$.
    \item For each positive integer $n$, let $\mu(n)$ be the M\"obius function, that is, $\mu(n)=(-1)^{\omega(n)}$ if $n$ is square-free, and $\mu(n)=0$ otherwise. 
\end{itemize}

\medskip

We need the following explicit bounds on $\tau(n)$ and $\varphi(n)$.

\begin{lemma}\label{lem:explicit}
Let $n\geq 3$ be an integer. Then we have
$$
\tau(n)\leq n^{\frac{1.5379\log 2}{\log \log n}},
$$
and
\[
\varphi(n) \geq \frac{n}{e^{\gamma}\,\log\log n \;+\; \frac{3}{\log\log n}},
\]
where  $\gamma$ is the Euler--Mascheroni constant.
\end{lemma}
\begin{proof}
The first estimate for the divisor function is due to Nicolas and Robin \cite{NR83}, and the second estimate for the Euler totient function is due to  Rosser and Schoenfeld \cite[Theorem 15]{RS62}.
\end{proof}

We also recall Weil's bound for complete character sums (see, for example \cite[Theorem 5.41]{LN97}). We will use Weil's bound to estimate character sums over the set of primitive roots $\mathcal{G}(p)$.

\begin{lemma}[Weil's bound]\label{lem:Weil}
Let $p$ be a prime. Let $\chi$ be a multiplicative character of $\F_p$ of order $k>1$, and let $f \in \F_p[x]$ be a monic polynomial that is not a $k$-th power of a polynomial in $\F_p[x]$. 
Let $s$ be the number of distinct roots of $f$ in the algebraic closure of $\F_p$. Then for each $c \in \F_p$, 
$$\bigg |\sum_{x\in\mathbb{F}_p}\chi\big(cf(x)\big) \bigg|\le(s-1)\sqrt p.$$
\end{lemma}

\subsection{Character sum estimates}
In this section, we prove the following proposition, which will be used in the proof of our main theorem.
\begin{proposition}\label{prop:charsum}
Let $p \equiv 1 \pmod 8$ be a prime. If
\begin{equation}\label{eq:omega}
\varphi(p-1)>2^{\omega(p-1)}(6\sqrt{p}+3),
\end{equation}    
then there exists a primitive root $a \in \F_p$ such that $a(a^2-a+1)$ is a nonzero $4$-th power.
\end{proposition}

Before we prove Proposition~\ref{prop:charsum}, we first point out that it immediately implies Theorem~\ref{thm:main} for all sufficiently large primes $p$. Indeed, as $p \to \infty$, Lemma~\ref{lem:explicit} implies that $\varphi(p-1)=p^{1-o(1)}$ and $2^{\omega(p-1)}\leq \tau(p-1)=p^{o(1)}$, and thus inequality~\eqref{eq:omega} holds for all sufficiently large primes $p$. Combining Propositions~\ref{prop:main} and~\ref{prop:charsum}, this yields Theorem~\ref{thm:main} for all sufficiently large primes $p$. Handling small primes requires some extra analysis, and that will be discussed in Section~\ref{subsec:finish}.

Next, we use character sum estimates to prove Proposition~\ref{prop:charsum}.
\begin{proof}[Proof of Proposition~\ref{prop:charsum}]
Let $\psi$ be a multiplicative character of order $4$ of $\F_p$. 
To show that there exists a primitive root $a$ such that $a(a^2-a+1)$ is a $4$-th power, by the orthogonality relation, it is equivalent to show the following:
$$
\sum_{j=0}^{3} \sum_{g \in \mathcal{G}(p)} \psi^j\big(g(g^2-g+1)\big)>0. 
$$
Thus, it suffices to show the following inequality:
\begin{equation}\label{eq:lb}
 \varphi(p-1)-\sum_{j=1}^{3} \bigg|\sum_{g \in \mathcal{G}(p)} \psi^j\big(g(g^2-g+1)\big)\bigg|>0.
\end{equation}    
Here we used the fact that $g(g^2-g+1)\ne0$ for every $g\in\mathcal G(p)$: since the roots of $x^2-x+1$ have order $6$, no primitive root of
$\F_p$ is a root of this polynomial when $p\equiv1\pmod8$.

To establish inequality~\eqref{eq:lb}, we employ Weil's bound to prove the following claim.
\begin{claim}
For $j\in \{1,2,3\}$, 
$$
\bigg|\sum_{g \in \mathcal{G}(p)} \psi^j\big(g(g^2-g+1)\big)\bigg| \leq 2^{\omega(p-1)} (2\sqrt{p}+1).
$$    
\end{claim}
\begin{poc}
Fix $j\in \{1,2,3\}$ and set $\chi=\psi^j$. Let $f(x)=x(x^2-x+1)$. Let $g_0$ be a fixed primitive root of $\mathbb{F}_p$. Then $$\mathcal{G}(p)=\{g_0^k: 1\leq k \leq p-1, \gcd(k,p-1)=1\}.$$
By M\"obius inversion, we have
\begin{align}
\sum_{g \in \mathcal{G}(p)} \chi(f(g))
&=\sum_{\substack{1\leq k \leq p-1 \\\gcd(k, p-1)=1}} \chi\left(f\left(g_0^k\right)\right) 
=\sum_{k=1}^{p-1} \bigg(\sum_{d \mid \gcd(k,p-1)} \mu(d)\bigg) \chi(f(g_0^k)) \notag\\
&=\sum_{d \mid p-1} \mu(d) \sum_{k=1}^{\frac{p-1}{d}} \chi\left(f\left(g_0^{k d}\right)\right) =\sum_{d \mid p-1} \frac{\mu(d)}{d} \sum_{x \in \mathbb{F}_p^*} \chi\left(f\left(x^d\right)\right).
\label{eq:d}
\end{align}

Let $d$ be a divisor of $p-1$. Note that the number of distinct roots of $f(x^d)=x^d(x^{2d}-x^d+1)$ over the algebraic closure of $\F_p$ is at most $2d+1$. Since $p>3$, it is easy to verify that the only possible multiple root of $f(x^d)$ is $0$, and in particular $f(x^d)$ is not a constant multiple of a square of a polynomial in $\F_p[x]$. Thus, Weil's bound (Lemma~\ref{lem:Weil}) implies that
$$
\frac{1}{d}\bigg| \sum_{x \in \mathbb{F}_p^*} \chi\left(f\left(x^d\right)\right) \bigg|\leq \frac{2d\sqrt{p}+1}{d}\leq 2\sqrt{p}+1.
$$

Thus, equation~\eqref{eq:d} implies that
$$
\bigg|\sum_{g \in \mathcal{G}(p)} \chi(f(g))\bigg|\leq \sum_{d \mid p-1} \mu^2(d) \bigg|\frac{1}{d} \sum_{x \in \mathbb{F}_p^*} \chi\left(f\left(x^d\right)\right) \bigg|\leq 2^{\omega(p-1)} (2\sqrt{p}+1),
$$
since the number of square-free divisors of $p-1$ is precisely $2^{\omega(p-1)}$, as required.
\end{poc}

Combining inequality~\eqref{eq:lb} and the above claim finishes the proof of the proposition.
\end{proof}

\subsection{Completion of the proof}\label{subsec:finish}

Our goal is to prove Theorem~\ref{thm:main} for all primes $p \equiv 1 \pmod 8$. Recall that inequality~\eqref{eq:omega} holds for all sufficiently large primes $p$.
To obtain a fully explicit result for all primes,
we next derive explicit bounds in \textbf{Case 1} and \textbf{Case 2}, and verify the remaining primes computationally in \textbf{Case 3}.

\medskip

\textbf{Case 1: $p\geq 7 \times 10^8$.}\\
In this case, we apply the crude bound that
$$
2^{\omega(p-1)}\leq \frac{\tau(p-1)}{2}.
$$
To see this, we can partition divisors of $(p-1)$ according to the largest odd factor of each divisor. Note that those divisors with the largest odd factor $d$ include $d,2d,4d,8d$ and potentially more; however, among them, only $d,2d$ can be square-free. Thus, by Propositions~\ref{prop:main} and~\ref{prop:charsum}, it suffices to show that
$$
\varphi(p-1)>\tau(p-1) \cdot (3\sqrt{p}+2)
$$
In view of Lemma~\ref{lem:explicit}, it suffices to verify the following inequality
$$
\frac{p-1}{e^{\gamma}\,\log\log (p-1) \;+\; \frac{3}{\log\log (p-1)}}>(3\sqrt{p}+2)\cdot (p-1)^{\frac{1.5379\log 2}{\log \log (p-1)}}.
$$  
This holds when $p\geq 7 \times 10^8$. Indeed, a direct differentiation shows that the ratio of the left-hand side to the right-hand side is increasing for $p\ge 7\times10^8$, and a numerical evaluation at $p=7\times10^8$ shows that this ratio exceeds $1$.
\medskip

\textbf{Case 2: $9.2 \times 10^7 \leq p\leq 7 \times 10^8$.}\\
In this case, we claim that $\omega(p-1)\leq 8$. Suppose, to the contrary, that $\omega(p-1)\ge9$. Since $p\equiv 1 \pmod 8$, we must have
$$
p\geq 8 \cdot 3 \cdot 5 \cdot 7 \cdot 11 \cdot 13 \cdot 17 \cdot 19 \cdot 23>8 \times 10^8,
$$
a contradiction. Then, $2^{\omega(p-1)}\le 2^8=256$.
Thus, by Lemma~\ref{lem:explicit} and Propositions~\ref{prop:main} and~\ref{prop:charsum}, it suffices to verify the inequality 
$$
\frac{p-1}{e^{\gamma}\,\log\log (p-1) \;+\; \frac{3}{\log\log (p-1)}}>256(6\sqrt{p}+3).
$$  
This holds when $p\geq 9.2 \times 10^7$.

\medskip

\textbf{Case 3: $p\leq 9.2 \times 10^7$.}\\
For each of the 1,331,027 primes $p\leq 9.2 \times 10^7$ with $p \equiv 1 \pmod 8$, we verified the hypothesis of Proposition~\ref{prop:main} by SageMath and we found that for each such prime $p$, there exists a primitive root $a\in \{2,3,\ldots, 339\}$ such that $a(a^2-a+1)$ is a nonzero $4$-th power. To achieve that, first we factorized $p-1$ to get the list of prime divisors of $p-1$. Note that, based on the analysis in Case 2, we have $\omega(p-1)\leq 8$. Then we enumerated all possible $a$ from $2$ to $p-1$, and stopped at the first desired $a$. To check that $a$ is a primitive root, it suffices to verify that $a^{\frac{p-1}{r}}\not \equiv 1 \pmod p$ for all the prime divisors $r$ of $p-1$ (there are at most $8$ such $r$). On the other hand, to check that $a(a^2-a+1)$ is a nonzero $4$-th power, it suffices to verify that $a^{(p-1)/4}(a^2-a+1)^{(p-1)/4}\equiv 1 \pmod p$. In total, the running time of our code \cite{Code} was less than $15$ minutes. 

\section{Problem~2.7 for finite abelian groups}\label{sec:problem27}

Throughout this section, let $G$ be a finite abelian group of odd order, written additively, and let $G=A\sqcup B$. For a permutation $\pi$ of $G$, put $\delta_\pi(x):=\pi(x)-x$.
Recall that $\pi$ is an \emph{orthomorphism} if $\delta_\pi$ is also a permutation of $G$. Let $\mathcal O(A,B)$ be the set of orthomorphisms $\pi$ of $G$ such that $\pi(A)=A$ and $\pi(B)=B$.

We first reduce the problem to showing that $|\mathcal O(A,B)|$ is odd.

\begin{lemma}\label{lem:orth-reduction}
If $|\mathcal O(A,B)|$ is odd, then Problem~\ref{prob:general} has an affirmative answer.
\end{lemma}

\begin{proof}
If $\pi\in\mathcal O(A,B)$, then $\pi^{-1}$ preserves both $A$ and $B$, and $\delta_{\pi^{-1}}=-\delta_\pi\circ\pi^{-1}$.
Hence $\delta_{\pi^{-1}}$ is a permutation of $G$, so $\pi^{-1}\in\mathcal O(A,B)$. Thus inversion defines an involution on the finite set $\mathcal O(A,B)$. If $|\mathcal O(A,B)|$ is odd, the inversion map has a fixed point; choose $\pi\in\mathcal O(A,B)$ with $\pi=\pi^{-1}$, so that $\pi^2=\mathrm{id}$.

Since $\delta_\pi$ is a permutation of $G$, there is a unique $z\in G$ with $\delta_\pi(z)=0$, and hence $z$ is the unique fixed point of $\pi$. Every other orbit of $\pi$ is therefore a $2$-cycle. Since $\pi$ preserves both $A$ and $B$, these $(|G|-1)/2$ two-element orbits give pairwise disjoint edges, each lying entirely in $A$ or entirely in $B$. Moreover, $\delta_\pi$ maps $G\setminus\{z\}$ bijectively onto $G\setminus\{0\}$. If $\{x,y\}$ is a $2$-cycle, then $\delta_\pi(x)=y-x$ and $\delta_\pi(y)=x-y$. Hence the two lengths $\pm(x-y)$ of the resulting edges run through every nonzero element of $G$ exactly once.
\end{proof}

It therefore remains to show that $|\mathcal O(A,B)|$ is odd. To detect this parity, we work in characteristic $2$. Put $\mathbb K:=\overline{\F_2}$, the algebraic closure of $\F_2$. For each $g\in G$, let $z_g$ be a commuting indeterminate over $\mathbb K$. Define the commutative $\mathbb K$-algebra
\[
\mathcal R:=\mathbb K[z_g:g\in G]/(z_g^2:g\in G),
\]
and continue to write $z_g$ for the image of the corresponding indeterminate in $\mathcal R$. The monomials $\prod_{g\in S}z_g$, with $S\subseteq G$, form a $\mathbb K$-basis of $\mathcal R$. In particular, $\prod_{g\in G}z_g\neq0$.

Fix an ordering of the elements of $G$, and give each subset the induced order. Put $T=(z_{y-x})_{x,y\in G}$.
Let $T_A$ and $T_B$ denote the principal submatrices of $T$ indexed by $A$ and $B$, respectively. We use the usual conventions that the determinant of the empty matrix and the empty product are both $1$. This convention also covers the cases in which $A$ or $B$ is empty.

We evaluate $\det(T_A)\det(T_B)$ in two ways: first in terms of $|\mathcal O(A,B)|$, and then by using the characters of $G$.

\subsection{A determinant formula for orthomorphisms}

\begin{lemma}\label{lem:det-encoding}
In $\mathcal R$,
\[
\det(T_A)\det(T_B)
=|\mathcal O(A,B)|\prod_{g\in G}z_g,
\]
where $|\mathcal O(A,B)|$ is viewed as an element of $\mathbb K$.
\end{lemma}

\begin{proof}
For a finite set $S$, let $\operatorname{Sym}(S)$ denote its group of permutations. Since $\operatorname{char}\mathbb K=2$, the signs in the determinant expansions disappear, and therefore
\[
\det(T_A)=\sum_{\sigma\in\operatorname{Sym}(A)}\prod_{x\in A}z_{\sigma(x)-x},
\qquad
\det(T_B)=\sum_{\tau\in\operatorname{Sym}(B)}\prod_{x\in B}z_{\tau(x)-x}.
\]
A pair $(\sigma,\tau)\in\operatorname{Sym}(A)\times\operatorname{Sym}(B)$ is equivalent to a permutation $\pi$ of $G$ preserving both $A$ and $B$. Hence
\[
\det(T_A)\det(T_B)
=\sum_{\substack{\pi\in\operatorname{Sym}(G)\\
                  \pi(A)=A,\ \pi(B)=B}}
  \prod_{x\in G}z_{\pi(x)-x}.
\]

For a fixed $\pi \in\operatorname{Sym}(G)$, this monomial is nonzero in $\mathcal R$ if and only if the differences $\pi(x)-x$, $x\in G$, are pairwise distinct. Indeed, a repeated difference produces a factor $z_g^2=0$, whereas distinct differences give one of the basis monomials above. Since $G$ is finite, pairwise distinctness is equivalent to $\delta_\pi$ being a permutation of $G$. Thus the nonzero terms are precisely those indexed by $\pi\in\mathcal O(A,B)$. For such a $\pi$, the values $\pi(x)-x$ run through every element of $G$ exactly once, so its contribution is $\prod_{g\in G}z_g$. Summing these contributions proves the lemma.
\end{proof}

\subsection{Evaluation via characters}

Let $\widehat G:=\operatorname{Hom}(G,\mathbb K^\times)$ be the character group of $G$. First we justify the orthogonality relations.

\begin{lemma}\label{lem:character-orthogonality}
The groups $G$ and $\widehat G$ are isomorphic. Moreover, for every $g\in G$,
\[
\sum_{\chi\in\widehat G}\chi(g)
=\begin{cases}
1,&g=0,\\
0,&g\neq0.
\end{cases}
\]
\end{lemma}

\begin{proof}
By the fundamental theorem of finite abelian groups, we may write $G\cong C_{n_1}\times\cdots\times C_{n_r}$, where each $n_i$ is odd and $C_{n_i}$ is the cyclic group of order $n_i$. Since $\mathbb K$ is algebraically closed and $n_i$ is odd, the polynomial $X^{n_i}-1$ has $n_i$ distinct roots in $\mathbb K$, and these roots form a cyclic group of order $n_i$. Thus $\operatorname{Hom}(C_{n_i},\mathbb K^\times)\cong C_{n_i}$ for every $i$, and hence $\widehat G\cong G$. This description also shows that the characters separate points: for every $g\neq0$, there is a character $\psi$ with $\psi(g)\neq1$.

If $g=0$, the sum is $|\widehat G|=|G|=1$ in characteristic $2$. If $g\neq0$, multiplication by such a $\psi$ permutes $\widehat G$, so $S_g:=\sum_{\chi\in\widehat G}\chi(g)$ satisfies $S_g=\psi(g)S_g$. Hence $S_g=0$.
\end{proof}

For the remainder of this subsection, if $M$ is a matrix with ordered row and column index sets, write $M[I,J]$ for the submatrix with row set $I$ and column set $J$. We will use the following two standard determinant identities. Suppose that $U$ has row set $I$ and column set $J$, while $V$ has row set $J$ and column set $I$, where $|I|\le |J|$, and that the entries lie in a commutative ring. Then the Cauchy--Binet formula gives
\[
\det(UV)=\sum_{\substack{S\subseteq J\\|S|=|I|}}
\det U[I,S] \, \det V[S,I].
\]
We will also use Jacobi's complementary-minor identity. Let $M$ be an invertible matrix over a field of characteristic $2$, and let $I,J$ be equally sized subsets of its row and column index sets. Then
\[
\det M^{-1}[J,I]
=\frac{\det M[I^c,J^c]}{\det M},
\]
where $I^c$ and $J^c$ denote the complements in the full row and column index sets, respectively. See \cite[Sections~0.8.4 and~0.8.7]{HJ}.

\begin{lemma}\label{lem:det-evaluation}
In $\mathcal R$,
\[
\det(T_A)\det(T_B)=\prod_{g\in G}z_g.
\]
\end{lemma}

\begin{proof}
Fix an ordering of the elements of $\widehat G$ and let $F=(\chi(x))_{x\in G,\,\chi\in\widehat G}$ be the character table of $G$. For $x,y\in G$, the $(x,y)$-entry of $FF^T$ is $\sum_{\chi\in\widehat G}\chi(x+y)$. Thus Lemma~\ref{lem:character-orthogonality} implies that $FF^T$ is the permutation matrix corresponding to the map $x\mapsto -x$. Its determinant is $1$ in characteristic $2$, so $(\det F)^2=1$ and hence $\det F=1$.

For each $\chi\in\widehat G$, put $f_\chi:=\sum_{g\in G}\chi(g)z_g\in\mathcal R$.
Because $\operatorname{char}\mathbb K=2$ and $z_g^2=0$, we have $f_\chi^2=\sum_{g\in G}\chi(g)^2z_g^2=0$. We will also use
\begin{equation}\label{eq:f-product}
\prod_{\chi\in\widehat G}f_\chi=\prod_{g\in G}z_g.
\end{equation}
To see this, expand the product on the left. A term survives only if every variable $z_g$ occurs exactly once. The surviving terms are indexed by bijections $\widehat G\to G$, and their total coefficient is $\det F^T=1$, since determinant signs disappear in characteristic $2$. This proves equation~\eqref{eq:f-product}.

We next diagonalize $T=(z_{y-x})_{x,y\in G}$. Let $D:=\operatorname{diag}(f_\chi:\chi\in\widehat G)$. For $x\in G$ and $\chi\in\widehat G$,
\[
(TF)_{x,\chi}
=\sum_{y\in G}z_{y-x}\chi(y)
=\chi(x)\sum_{g\in G}\chi(g)z_g
=\chi(x)f_\chi.
\]
Hence $TF=FD$, and therefore
\begin{equation}\label{eq:T-diagonal}
T=FDF^{-1}.
\end{equation}

It remains to evaluate the principal minors $T_A$ and $T_B$.
For $S\subseteq\widehat G$, write
\[
S^c:=\widehat G\setminus S,\qquad
f_S:=\prod_{\chi\in S}f_\chi.
\]
If $|S|=|A|$, also put
\[
d_S:=\det F[A,S]\det F[B,S^c].
\]
From $T=FDF^{-1}$ we have
\[
T_A=F[A,\widehat G]D F^{-1}[\widehat G,A].
\]
Thus Cauchy--Binet gives
\[
\det(T_A)
=\sum_{\substack{S\subseteq\widehat G\\|S|=|A|}}
 \det F[A,S]\det F^{-1}[S,A]f_S.
\]
By Jacobi's identity and $\det F=1$,
\[
\det F^{-1}[S,A]=\det F[B,S^c],
\]
and therefore
\begin{equation}\label{eq:detA}
\det(T_A)
=\sum_{\substack{S\subseteq\widehat G\\|S|=|A|}}d_Sf_S.
\end{equation}
Similarly, Cauchy--Binet and Jacobi give
\[
\det(T_B)
=\sum_{\substack{R\subseteq\widehat G\\|R|=|B|}}
 \det F[B,R]\det F[A,R^c]f_R.
\]
Replacing $R$ by $S^c$ yields
\begin{equation}\label{eq:detB}
\det(T_B)
=\sum_{\substack{S\subseteq\widehat G\\|S|=|A|}}d_Sf_{S^c}.
\end{equation}

Multiplying equations~\eqref{eq:detA} and~\eqref{eq:detB} and expanding, a term
$d_Sd_{S'}f_Sf_{(S')^c}$ vanishes unless $S\subseteq S'$, since
$f_\chi^2=0$. As $|S|=|S'|$, this forces $S=S'$. Hence
\[
\det(T_A)\det(T_B)
=f_{\widehat G}
 \sum_{\substack{S\subseteq\widehat G\\|S|=|A|}}d_S^2
=f_{\widehat G}
 \left(
 \sum_{\substack{S\subseteq\widehat G\\|S|=|A|}}d_S
 \right)^2.
\]
By the Laplace expansion of $\det F$ along the rows indexed by $A$,
the sum in parentheses is $\det F=1$; as before, the signs disappear
in characteristic $2$. Thus, by equation~\eqref{eq:f-product},
\[
\det(T_A)\det(T_B)=f_{\widehat G}
=\prod_{g\in G}z_g. \qedhere
\]
\end{proof}

Comparing the coefficients of the basis monomial $\prod_{g\in G}z_g$ in Lemmas~\ref{lem:det-encoding} and~\ref{lem:det-evaluation} gives $|\mathcal O(A,B)|=1$ in $\mathbb K$. Thus $|\mathcal O(A,B)|$ is odd, and Theorem~\ref{thm:general} follows from Lemma~\ref{lem:orth-reduction}.

\section*{Acknowledgments}
The authors thank Muhammad Afifurrahman, Keith Ball, and Laurence P. Wijaya for helpful discussions. C.H. Yip thanks the Institute for Basic Science for its hospitality during his visit, where this project was initiated. S. Yoo was supported by the Institute for Basic Science (IBS-R029-C1). S. Yu was supported by the National Natural Science Foundation of China (NSFC) under Grant No. 12271023.

\bibliographystyle{abbrv}
\bibliography{references}

\begin{thebibliography}{10}

\bibitem{A25}
M.~Afifurrahman, D.~I.~D. Primaskun, P.~Etriana~Putri, and L.~P. Wijaya.
\newblock On constructing 1-factors of labelled complete graph.
\newblock {\em Bol. Soc. Mat. Mex. (3)}, 31(2):Paper No. 82, 10, 2025.

\bibitem{BOP19}
K.~Ball, O.~Ortega-Moreno, and M.~Prodromou.
\newblock Hadamard matrices and 1-factorizations of complete graphs.
\newblock {\em Mathematika}, 65(3):488--499, 2019.

\bibitem{CD06}
C.~J. Colbourn and J.~H. Dinitz, editors.
\newblock {\em Handbook of combinatorial designs}.
\newblock Discrete Mathematics and its Applications (Boca Raton). Chapman \& Hall/CRC, Boca Raton, FL, second edition, 2007.

\bibitem{Evans18}
A.~B. Evans.
\newblock {\em Orthogonal Latin squares based on groups}, volume~57 of {\em Developments in Mathematics}.
\newblock Springer, Cham, 2018.

\bibitem{H07}
K.~J. Horadam.
\newblock {\em Hadamard matrices and their applications}.
\newblock Princeton University Press, Princeton, NJ, 2007.

\bibitem{HJ}
R.~A. Horn and C.~R. Johnson.
\newblock {\em Matrix Analysis}.
\newblock Cambridge University Press, Cambridge, second edition, 2013.

\bibitem{JDM61}
D.~M. Johnson, A.~L. Dulmage, and N.~S. Mendelsohn.
\newblock Orthomorphisms of groups and orthogonal {L}atin squares. {I}.
\newblock {\em Canad. J. Math.}, 13:356--372, 1961.

\bibitem{LN97}
R.~Lidl and H.~Niederreiter.
\newblock {\em Finite fields}, volume~20 of {\em Encyclopedia of Mathematics and its Applications}.
\newblock Cambridge University Press, Cambridge, second edition, 1997.

\bibitem{MR85}
E.~Mendelsohn and A.~Rosa.
\newblock One-factorizations of the complete graph---a survey.
\newblock {\em J. Graph Theory}, 9(1):43--65, 1985.

\bibitem{NR83}
J.-L. Nicolas and G.~Robin.
\newblock Majorations explicites pour le nombre de diviseurs de {$N$}.
\newblock {\em Canad. Math. Bull.}, 26(4):485--492, 1983.

\bibitem{P33}
R.~E. A.~C. Paley.
\newblock On orthogonal matrices.
\newblock {\em J. Math. Phys., Mass. Inst. Techn.}, 12:311--320, 1933.

\bibitem{RS62}
J.~B. Rosser and L.~Schoenfeld.
\newblock Approximate formulas for some functions of prime numbers.
\newblock {\em Illinois J. Math.}, 6:64--94, 1962.

\bibitem{W97}
W.~D. Wallis.
\newblock {\em One-factorizations}, volume 390 of {\em Mathematics and its Applications}.
\newblock Kluwer Academic Publishers Group, Dordrecht, 1997.

\bibitem{Code}
C.~H. Yip, S.~Yoo, and S.~Yu.
\newblock Paley-type matrices and 1-factorizations of complete graphs: Github repository, 2026.
\newblock \href{https://github.com/kyleyip111/Paley-type-matrices-and-1-factorizations-of-complete-graphs}{github.com/kyleyip111/Paley-type-matrices-and-1-factorizations-of-complete-graphs}.

\end{thebibliography}

\end{document}